\documentclass{elsarticle}
\usepackage[linesnumbered, ruled, vlined]{algorithm2e}
\usepackage{ amsmath, amsthm, amssymb, enumerate, amsfonts, xypic, hyperref}

\SetKwInOut{Input}{Input}
\SetKwInOut{Output}{Output}

\def\ZZ{{\mathbb Z}}

\def\FF{{\mathbb F}}

\def\RR{{\mathbb R}}
\def\QQ{{\mathbb Q}}

\def\OO{{\mathcal O}}

\def\pp{{\mathfrak p}}
\def\qq{{\mathfrak q}}
\def\aa{{\mathfrak a}}

\def\bb{{\mathfrak b}}
\def\mm{{\mathfrak m}}
\def\ff{{\mathfrak f}}
\def\vv{{\mathfrak v}}
\def\uu{{\mathfrak u}}
\def\frakell{{\mathfrak{l}}}
\def\LL{{\mathfrak{L}}}
\def\frakC{{\mathfrak{C}}}

\def\disc{{\text{\textnormal{disc}}}}
\def\Gal{{\text{\textnormal{Gal}}}}
\def\cl{{\text{\textnormal{Cl}}}}
\def\Cl{{\text{\textnormal{Cl}}}}

\newcommand{\End}[1]{\text{\textnormal{End}}(#1)}

\newtheorem{theorem}{Theorem}[section]

\newtheorem{lemma}[theorem]{Lemma}
\newtheorem{prop}[theorem]{Proposition}
\newtheorem{corollary}[theorem]{Corollary}
\newtheorem{definition}[theorem]{Definition}

\begin{document}
\title{Computing the endomorphism ring of an ordinary abelian surface over a finite field}

\author{Caleb Springer\fnref{fn1}}
\ead{cks5320@psu.edu}

\address{Department of Mathematics, The Pennsylvania State University, University Park,
PA 16802, USA}
\fntext[fn1]{The author was partially supported by National Science Foundation grants DMS-1056703
and CNS-1617802.}
\date{\today}

\begin{abstract}
 We present a new algorithm for computing the endomorphism ring of an ordinary abelian surface over a finite field which is subexponential and generalizes an algorithm of Bisson and Sutherland for elliptic curves.
  The correctness of this algorithm only requires the heuristic assumptions required by the  algorithm of Biasse and Fieker \cite{bf} which computes the class group of an order in a number field in subexponential time. Thus we avoid the multiple heuristic assumptions on isogeny graphs and polarized class groups which were previously required. The output of the algorithm is an ideal in the maximal totally real subfield of the endomorphism algebra, generalizing the elliptic curve case.
\end{abstract}

\maketitle

\section{Introduction}
Computing the endomorphism ring of an abelian variety is a fundamental problem of computational number theory with applications in cryptography. 
Consider the foundational case when $E$ is an ordinary elliptic curve over $\FF_q$. In this case, $\End E$ is isomorphic to an order in the quadratic imaginary field $K = \QQ(\pi)$ where $\pi$ is the Frobenius of $E$. Further, the orders in $K$ which contain $\pi$ are precisely the orders which arise as the endomorphism ring of an elliptic curve $E'$ over $\FF_q$ which is isogenous to $E$ \cite[Theorem 4.2]{waterhouse}.  Because $K$ is a quadratic imaginary field, the orders of $K$ are uniquely identified by their index in $\OO_K$. Thus, computing $\End E$ is equivalent to computing the index $[\OO_K : \End E]$, which is a divisor of $[\OO_K : \ZZ[\pi]]$.

The first algorithm for computing $\End E$ for an ordinary elliptic curve $E$ was given by Kohel in his thesis \cite{kohel}.  Kohel's algorithm is deterministic and has exponential expected runtime $O(q^{\epsilon + 1/3})$, assuming GRH.  The key fact needed for this algorithm was Kohel's discovery that certain $\ell$-isogeny graphs (i.e., graphs whose vertices can be identified with isomorphism classes of elliptic curves and whose edges can be identified with isogenies of prime degree $\ell$) have a special ``volcano" structure.
By navigating the volcano graphs, one may deduce the prime factors of $[\OO_K : \End E]$, and thus determine $\End E$ as desired, although this method is slow when working with large prime  factors.

To make a faster algorithm for computing $\End E$, Bisson and Sutherland exploited the following fact, which was proven by Waterhouse \cite[Theorem 4.5]{waterhouse}.   Given an order $\OO\subseteq K$, the class group $\Cl(\OO)$ acts faithfully on the set of isomorphism classes of elliptic curves isogenous to $E$ with endomorphism ring isomorphic to $\OO$, where an ideal of norm $\ell$ acts via an isogeny of degree $\ell$.
Thus, one may deduce information about the class group of $\End E$ by computing various isogenies.  In particular, one may compare $\cl(\End E)$ to $\cl(\OO)$ for known ``testing" orders $\OO\subseteq K$ by simply computing class group relations.  Bisson and Sutherland prove that determining which relations hold in $\cl(\End E)$ is sufficient for determining the large prime factors of ${[\OO_K : \End E]}$, which leads to a probabilistic algorithm to determine $[\OO_K : \End E]$ in subexponential time
 \cite{bisson-sutherland}.  

Bisson extended the algorithm above to absolutely simple, principally polarized, ordinary abelian varieties of dimension 2 in \cite{bisson}. The endomorphism ring of such an abelian variety is an order in a quartic CM field.  Because the isogeny graph structure of ordinary abelian varieties of dimension 2 was essentially unknown at the time, the correctness of Bisson's algorithm required multiple heuristic assumptions about the relevant isogeny graphs and polarized class groups.  
Additionally, these orders are no longer uniquely identified by their index, as in the elliptic curve case.  Instead, Bisson identifies an order $\OO$ by a ``lattice of relations" which hold in $\OO$.

This paper gives a different generalization of Bisson and Sutherland's elliptic curve algorithm.  Our algorithm identifies orders by ideals in the maximal totally real subfield of the endomorphism algebra, and its correctness only relies on the assumptions required for computing the class group of an order in a number field in subexponential time, as in \cite{bf}.    As a trade-off, we must explicitly restrict the class of abelian varieties that we consider.

\subsection{Main Result}
Let $A$ be an absolutely simple, principally polarized, ordinary abelian variety of dimension 2 defined over a finite field $\FF_q$ with Frobenius $\pi$.  Then $\End A$ is an order containing $\ZZ[\pi, \overline\pi]$ in the quartic CM field $K = \QQ(\pi)$.  We will assume that $A$ has maximal RM, or maximal \emph{real multiplication}, which means that $\End A \supseteq \OO_F$ where $F$ is the maximal totally real subfield of $K$.  According to Brooks, Jetchev and Wesolowski \cite[Theorem 2.1]{bjw}, the orders $\OO\subseteq K$ which contain $\OO_F$ are in bijective correspondence with ideals of $\OO_F$, where $\OO$ is associated to $\ff\cap \OO_F$ and $\ff$ is the conductor ideal of $\OO$.  This allows our algorithm to have a simple output, namely the ideal of $\OO_F$ which uniquely identifies $\End A$.

The foundation of our algorithm is the free action of the ideal class group $\cl(\OO)$ on the set of isomorphism classes of abelian varieties isogenous to $A$ with endomorphism ring $\OO\subseteq K$.  Just like in the elliptic curve case, this means that a product of ideals is trivial in $\cl(\End A)$ if and only if the corresponding isogeny maps $A$ to an isomorphic abelian variety.  Using the same tactic as Bisson and Sutherland, we compare $\cl(\End A)$ to $\cl(\OO)$ for known ``testing" orders $\OO\subseteq K$ by generating class group relations. When $\End A$ contains $\OO_F$, we show through class field theory that class group relations are sufficient for determining the large prime factors of the ideal $\ff^+$ which identifies the endomorphism ring $\End A$.  

However, unlike the elliptic curve case, it is not always possible to compute the action of $\cl(\OO)$ on principally polarized abelian varieties of dimension $g > 1$.  In fact, the isogenies corresponding to elements of $\cl(\OO)$ do not even preserve principal polarizability in general.  Bisson avoided this obstruction for his abelian surface algorithm by instead working with a convenient subgroup of the \emph{polarized class group} $\frakC(\OO)$, which has a nicer action that respects polarizations. Unfortunately, the structure of $\frakC(\OO)$ cannot be analyzed through machinery like class field theory in the same way that we can analyze $\cl(\OO)$, so it is difficult to demonstrate the existence of sufficiently many relations in $\frakC(\OO)$.

For our algorithm, we ensure that it is possible to compute the action of $\cl(\OO)$ by restricting our attention to the case when $F$ has narrow class number 1.  In this case, there is a surjective map from $\frakC(\OO)$ to $\cl(\OO)$ and the action of $\cl(\OO)$ preserves principal polarizability.  We also assume that the index $[\OO_F : \ZZ[\pi + \overline \pi]]$ is not even, and that $\OO_K^* = \OO_F^*$ so that all isogenies are computable via Cosset, Robert, Dudeanu, et al. \cite{cr, djrv}, and certain isogeny graphs are volcanoes \cite{bjw}.  This latter fact allows us to determine the small prime factors of $\ff^+(A)$, much like in
 the elliptic curve case.

In summary, we focus our attention on absolutely simple, principally polarized, ordinary abelian surfaces $A$ for which the following is true:
 \begin{enumerate}
	 \item $A$ has maximal real multiplication by $F$.
	 \item $\OO_K^* = \OO_F^*$.
	 \item $F$ has narrow class number 1.
	 \item The conductor gap $[\OO_F : \ZZ[\pi + \overline \pi]]$ is not divisible by 2.
 \end{enumerate}
 
\begin{theorem}\label{main-thm}
There is a subexponential algorithm which, given an abelian variety $A$ of dimension $2$ satisfying the conditions above, outputs the ideal of $\OO_F$ uniquely identifying $\End A$.
 \end{theorem}

 Note that the restrictions imposed by Theorem \ref{main-thm} make this algorithm considerably less general than the algorithm of Bisson, which accepts all absolutely simple, principally polarized, ordinary abelian surfaces, with only the exclusion of a certain zero-density set of worst-case varieties. However, our extra restrictions provide benefits as a trade-off.  First, the correctness of our algorithm relies solely on the heuristic assumptions required for Biasse and Fieker's subexponential algorithm for solving the principal ideal problem \cite{biasse, bf}, thereby avoiding the unconventional heuristic assumptions needed for Bisson's algorithm.
 Additionally, the zero-density set excluded by Bisson's algorithm is not explicitly known, while the conditions for Theorem \ref{main-thm} are explicit and verifiable.
 
With the notation,
$$
	L[a,c](n) = \exp\left( (c+o(1))\log(n)^a(\log\log n)^{1-a}\right)	
$$ 
we heuristically bound the asymptotic runtime of the algorithm by
$$
	L[1/2, 2c](q)+ L[1/2, 2\sqrt{d+1}](q).
$$
where $c$ and $d$ are constants corresponding to the difficulty of the principal ideal problem \cite{biasse, bf} and isogeny computation \cite{cr, djrv}, respectively.  The algorithm can also be modified to produce a short certificate which allows a third party to verify the correctness of the output. This takes subexponential time using the same heuristic assumptions as before.

\subsection{Organization}
The paper is organized as follows.  In Sections 2 and 3, we present background material for abelian varieties over finite fields with maximal RM and prove that class group relations are sufficient to distinguish orders in CM fields. In Section 4, we restrict to the case of abelian varieties of dimension $g=2$ with the special properties outlined above, present the algorithm and analyze the runtime under certain heuristic assumptions.  In Section 5, we present an illustrative example.

\subsection{Acknowledgements}
The author would like to thank Benjamin Wesolowski for his helpful discussions on isogeny graphs, and for explaining a proof for Lemma \ref{deg-frak-iso}. The author also thanks David Kohel for his suggested potential improvement of the main theorem, and Kirsten Eisentr\"ager, Gaetan Bisson and Andrew Sutherland for their many helpful comments.

\section{Background \label{prelim}}

\subsection{Notation}
Fix an absolutely simple, principally polarized, ordinary abelian variety $A$ of dimension $g$ over $\FF_q$ with Frobenius endomorphism $\pi$.  Using Pila's algorithm \cite{pila}, the characteristic polynomial $f_\pi$ of $\pi$ can be computed in time polynomial in $\log q$. The polynomial $f_\pi(x)\in \ZZ[x]$ encodes many features of $A$, such as the isogeny class, the number of $\FF_q$-rational points, and the fact that $A$ is ordinary \cite{tate}.  In the remainder of the paper, we will assume that $f_\pi$ is known and an embedding $\End A \hookrightarrow \QQ(\pi) \cong \QQ[x]/(f_\pi)$ is given for simplicity.  The field $\QQ(\pi)$ is a CM field of degree $2g$ which we denote by $K$.  Let $F$ denote its maximal totally real subfield.  We will write $\OO$ for an arbitrary order in $K$ which is possibly non-maximal.

\subsection{Identifying orders}
\label{identify-orders}

Waterhouse \cite[Theorem 7.4]{waterhouse} showed that orders of $K$ arising as endomorphism rings of abelian varieties isogenous to $A$ are precisely the orders $\OO$ satisfying $\ZZ[\pi, \overline \pi] \subseteq \OO \subseteq \OO_K$.  Our goal is to determine which order is $\End A$, although it is not clear how to uniquely identify such orders in a simple way.  In the elliptic curve case, $K$ is a quadratic imaginary field and for each positive integer $n$ there is precisely one order $\OO$ of index $n$ inside $\OO_K$, hence the orders of $K$ can be uniquely identified by positive integers. To identify orders in CM fields in general, we will restrict our attention to abelian varieties with \emph{maximal RM}, i.e. $\End A$ is an order containing $\OO_F[\pi]$.

  Recall that, given $\OO\subseteq \OO_K$, the \emph{conductor ideal} of $\OO$ is defined as the ideal
  $$
  	\ff = \{\alpha\in K : \alpha\OO_K \subseteq \OO\}.
  $$
  Notice that $\ff$ is the largest subset of $K$ which is an ideal in both $\OO_K$ and $\OO$.  When $K$ is a quadratic imaginary field, the conductor ideal $\ff$ of an order $\OO$ is the ideal of $K$ generated by the integer $[\OO_K : \OO]$.  Hence the following theorem is a generalization of the fact that orders in a quadratic imaginary field are uniquely determined by their index.

\begin{theorem}[Theorem 2.1, \cite{bjw}]
Orders in $K$ containing $\OO_F$ are in bijective correspondence with ideals of $\OO_F$, as follows.  To each order $\OO\subseteq K$ containing $\OO_F$, associate the ideal $\ff^+ = \ff \cap \OO_F$ where 
$\ff\subseteq \OO_K$ is the conductor ideal of $\OO$.  Similarly, to any ideal $\ff^+\subseteq \OO_F$, associate the order $\OO_F + \ff^+\OO_K$.

\label{ideal}
 \end{theorem}

 \begin{center}
 \begin{algorithm}[H]
\SetAlgoLined
\Input{An order $\OO$ in a CM field $K$ of degree $2g$ with maximal totally real subfield $F$ satisfying $\OO\cap F = \OO_F$, and $\ZZ$-bases $\{\alpha_1,\dots, \alpha_g\}$, $\{\tau_1,\dots, \tau_{2g}\}$ and $\{\omega_1,\dots, \omega_{2g}\}$  for $\OO_F$, $\OO$ and $\OO_K$, respectively.}
\Output{Ideal $\ff^+\subseteq \OO_F$ identifying the order $\OO\subseteq \OO_K$ containing $\OO_F$.}

	  Define $b_{i,j,k}\in \QQ$ by
	$
		\alpha_i\omega_j = \sum_{k = 1}^{2g} b_{i,j,k} \tau_k
	$\\
	
	 Define the matrix
	$$
		M := \left( \begin{array}{c} M_1\\ \vdots\\ M_{2g}\end{array}\right)
	$$
	where
	$$
		M_j := 
			\left( \begin{array}{ccc} 
				b_{1,j,1} & \dots & b_{g,j,1}\\ 
				\vdots & \ddots & \vdots \\  
				b_{1,j,2g} & \dots & b_{g,j,2g}
			\end{array}\right)
	$$\\
	
	 Let $d$ be the greatest common divisor of all $d'\in \ZZ$ with $d' M \in \ZZ^{(2g)^2 \times g}$;\label{gcd-step}\\
	 
	 Let $H\in \ZZ^{g\times g}$ be the row Hermite normal form of $dM$;\label{HNF}\\
	 
	 $\ff^+ = \ZZ\beta_1 + \dots + \ZZ\beta_g$ where $(\beta_1,\dots, \beta_g) = (\alpha_1,\dots, \alpha_g)dH^{-1}.$\label{HNF-inv}

  \caption{\textsc{Compute the identifying ideal}}
  \label{compute-ideal}
\end{algorithm}
\end{center}

This algorithm is a modification of the algorithm of Kl\"uners and Pauli \cite[\S 6]{kp}, and it allows one to compute the identifying $\OO_F$-ideal of any order $\OO\subseteq \OO_K$ containing $\OO_F$, thereby making the correspondence of Theorem \ref{ideal} computable.

\begin{prop} 
Algorithm \ref{compute-ideal} is correct.
 \end{prop}
 \begin{proof} 
	Let $\beta = \sum_{i = 1}^g a_i\alpha_i$ be an element of $F$.  By definition, $\beta$ is in $\ff^+$ if and only if $\beta\OO_K \subseteq \OO$, i.e. if and only if 
	$$
		\beta\omega_j 
			= \sum_{i = 1}^g a_i\alpha_i\omega_j
			= \sum_{i = 1}^g a_i \sum_{k = 1}^{2g}b_{i,j,k} \tau_k 
			= \sum_{k = 1}^{2g} \left(\sum_{i = 1}^g a_i b_{i,j,k} \right)\tau_k\in \OO
	$$
for every $1\leq j\leq 2g$.  Equivalently, $\sum_{i = 1}^g a_ib_{i,j,k}\in \ZZ$ for all $1\leq j\leq 2g$ because $\tau_k$ is a $\ZZ$-basis for the order $\OO$.

Define $M_j$ as the ``multiplication by $\omega_j$'' matrix:
$$
		M_j 
		:= \left( \begin{array}{ccc} b_{1,j,1} & \dots & b_{g,j,1}\\ \vdots & \ddots & \vdots \\  b_{1,j,2g} & \dots & b_{g,j,2g}\end{array}\right)
$$
and put
	$$
		M := \left( \begin{array}{c} M_1\\ \vdots\\ M_{2g}\end{array}\right)
	$$
	Then, $\beta \in \ff^+$ if and only if $M\vec{a} \in \ZZ^{2g\times 2g}$ where $\vec{a} = (a_1,\dots, a_g)^{tr}$.  Note that $M$ has maximal rank $g$ because each $M_j$ is injective (i.e. multiplication by $\omega_j$ has no kernel).

Write $d$ for the greatest common divisor of the integers $d'$ such that $d' M$ has integer coefficients.  Let $H$ be the row Hermite Normal Form of $dM$.  View $H$ as a $g\times g$ matrix by removing all of the zero rows.  Because $M$ has maximal rank $g$, $H$ is invertible. Then
$$
	\beta\in \ff^+ 
		\iff M\vec{a} \in \ZZ^{(2g)^2}
		\iff dM\vec{a} \in d\ZZ^{(2g)^2}
		\iff H\vec{a}\in d\ZZ^{g}
		\iff \vec{a}\in dH^{-1}\ZZ^{g}
$$
Because $\ff^+$ is an integral ideal, each $a_i$ must be an integer.  We deduce that $dH^{-1}$ is an integer matrix whose columns are a basis for $\ff^+$, written in the basis $\{\alpha_1, \dots, \alpha_g\}$.  Thus  we can write a basis of $\ff^+$ as
$$
	(\beta_1, \dots, \beta_g) = (\alpha_1, \dots, \alpha_g) dH^{-1}.
$$

 \end{proof}
 
Before  analyzing the running time of Algorithm \ref{compute-ideal}, we must choose integral bases which have multiplication tables of small size. Given a basis $b_1,\dots, b_n$ for an order of discriminant $\Delta$ in a number field of degree $n$, the \emph{multiplication table} $(x_{i,j,k})$ is defined by $b_ib_j = \sum_{k = 1}^n x_{i,j,k}b_k$.  There is a polynomial time basis reduction algorithm that provides a basis whose multiplication table has size $O(n^4(2 + \log|\Delta|))$; see \cite[\S2.10]{lenstra-algorithms} or \cite[\S5]{buch-pre}.  We apply this to $\OO_F$ and $\OO_K$ to get reduced bases $\{\alpha_1,\dots, \alpha_g\}$ and $\{\omega_1,\dots, \omega_{2g}\}$, respectively. Any $\OO\subset \OO_K$ can be represented by a basis $\{a_1\omega_1,\dots, a_{2g}\omega_{2g}\}$ where $|a_i| \leq [\OO_K : \OO]$.

\begin{prop} 
Assuming that the bases for $\OO_K$, $\OO$ and $\OO_F$ are chosen as outlined above, Algorithm \ref{compute-ideal} has running time polynomial in $g$ and $\log|\disc(\OO)|$.
 \end{prop}
 
  \begin{proof} 
Consider the integers $c_{i,j,k}$ defined by
 \begin{equation}
 \label{mult-table}
 	\alpha_i\omega_j = \sum_{k = 1}^{2g} c_{i,j,k}\omega_k.
 \end{equation}
We see that $c_{i,j,k} = a_kb_{i,j,k}$ by the choice of bases.  Write $b_{i,j,k}$ as a rational number in reduced form:
$$
	b_{i,j,k} = \frac{c_{i,j,k}}{a_k} = \frac{c_{i,j,k}/ \text{gcd}(c_{i,j,k}, a_k)}{a_k/ \text{gcd}(c_{i,j,k}, a_k)}.
$$
 Hence the integer $d$ used in Step \ref{gcd-step} is
$$
	d = \text{lcm}\left\{\frac{a_k}{\gcd(a_k,c_{i,j,k})}\right\}.
$$

Algorithm \ref{compute-ideal} essentially consists of computing the integer $d$ in Step \ref{gcd-step}, the Hermite Normal Form $H$ of the matrix $dM$ in Step \ref{HNF}, and the inverse $H^{-1}$ in Step \ref{HNF-inv}. Each of these steps are polynomial time in the dimension and the size of the coefficients. We will modify \cite[Proposition 5.3]{buch-pre} to bound the size of $c_{i,j,k}$ by a polynomial in $g$, $\log|\disc(\OO_K)|$ and $\log|\disc(\OO_F)|$.  Because $[\OO_K : \OO]$, $\log|\disc(\OO_K)|$ and $\log|\disc(\OO_F)|$ are all bounded by $\log |\disc(\OO)|$, this is sufficient for bounding the running time of the algorithm.

 Recall that there are $2g$ conjugate pairs of complex embeddings of $K$, which we denote 
 $
 	\sigma_1,\dots, \sigma_g,\overline\sigma_1, \dots,\overline\sigma_g
$ 
This provides an embedding $K \to \RR^{2g}$ via 
$$
	\alpha\mapsto 
		\underline \alpha :=
		(\text{Re}(\sigma_1(\alpha)), \text{Im}(\sigma_1(\alpha)), \dots, \text{Re}(\sigma_g(\alpha)), \text{Im}(\sigma_g(\alpha))).
$$
Embedding the equation \eqref{mult-table} gives the linear equation  
$$
	\underline{\alpha_i\omega_j} = \sum_{k = 1}^{2g} c_{i,j,k} \underline{\omega_k}.
$$
The matrix $(\underline\omega_k)_k$ has determinant $2^{-g}{ |\disc(\OO_K)|}^{1/2}$ by definition.  Hence by Cramer's rule
$$
	|c_{i,j,k}| 
		\leq \frac{\| \underline{\alpha_i\omega_j}\| \cdot \prod_{r\neq k} \| \underline{\omega_r}\| }{{ 2^{-g}|\disc(\OO_K)|}^{1/2}}
$$
Because the bases for $\OO_K$ and $\OO_F$ are reduced, the norms of the vectors appearing in the numerator are bounded.  Specifically, $\log\|\underline\omega_k\|$ and $\log\|\underline\alpha_i\|$ are bounded by a polynomial in $g$, $\log|\disc(\OO_K)|$ and $\log|\disc(\OO_F)|$; see \cite[Proposition 5.2]{buch-pre}.  We conclude that the sizes of the integers $c_{i,j,k}$ are bounded as desired.
  \end{proof}
  
For the purposes of this paper, we want bound the running time in terms of the size of the finite field $\FF_q$. This is achieved in the following corollary. Even though we will consider $g$ to be a constant in the remainder of the paper, we note that the algorithm is also polynomial in $g$.
  \begin{corollary} 
  \label{ideal-alg}
  Maintain the notation of the previous proposition, but additionally assume $K = \QQ(\pi)$ where $\pi$ is the Frobenius of a simple ordinary abelian variety $A$ of dimension $g$ over $\FF_q$.  If $\OO$ contains $\pi$, then Algorithm \ref{compute-ideal} computes the ideal $\ff^+$ identifying $\OO$ in time polynomial in $g$ and $\log q$.
   \end{corollary}
   \begin{proof} 
  Since the algorithm is proven to be polynomial in $g$ and $\log|\disc(\OO)|$, we simply need to notice that $\OO \supseteq \ZZ[\pi]$ implies
$$
	|\disc(\OO)| 
		\leq |\disc(\ZZ[\pi])|
		\leq (2\sqrt{q})^{2g(2g-1)}.
$$
The second part of the inequality follows immediately from the definition of discriminant because each root of the characteristic polynomial of $\pi$ has absolute value $\sqrt{q}$.
Upon taking logarithms, we obtain our result.
   \end{proof}
 
For convenient notation, denote the ideal identifying an order $\OO\supseteq \OO_F$ by $\ff^+(\OO)$, and denote the order identified by $\ff^+$ as $\OO(\ff^+)$.  Denote by $\ff^+(A)$ the ideal identifying the order isomorphic to $\End A$ under the fixed embedding $\End A \hookrightarrow K$.
Now computing $\End A$ is equivalent to computing $\ff^+(A)$, which is a divisor of  $\ff^+(\OO_F[\pi])$.  Hence we consider the factors of $\ff^+(\OO_F[\pi])$ and determine which prime powers divide $\ff^+(A)$. We will present different methods for dealing with small and large prime factors.

 \subsection{Class group action}
 Following Brooks, Jetchev and Wesolowski \cite{bjw}, we will focus on $\frakell$-isogenies, defined below.  Assume that $A$ has maximal RM and write $\vv$ for the ideal $\ff^+(\OO_F[\pi])$.  Given $\alpha\in \End A$, we write $A[\alpha]$ for the kernel of $\alpha$.
 
 \begin{definition}[$\frakell$-isogeny]
Let $\frakell\subseteq \OO_F$ be a prime ideal coprime to $\ff^+(A)$.  An $\frakell$-isogeny from A is an isogeny whose kernel is a proper subgroup of $A[\frakell] = \cap_{\alpha\in \frakell} A[\alpha]$ which is stable under the action of $\OO_F$.  
  \end{definition}
  
The following lemma was stated without proof in the definition of $\frakell$-isogeny in \cite{bjw}.  We record a proof here for completeness.\footnote{The author thanks Benjamin Wesolowski for explaining this proof in private communication.}
  
  \begin{lemma} 
  \label{deg-frak-iso}
 Using the notation above, the degree of  an $\frakell$-isogeny is the norm $N_{F/\QQ}(\frakell)$.
   \end{lemma}
    \begin{proof}  
    Recall that $\#A[\frakell] = N_{K/\QQ}(\frakell) = N_{F/\QQ}(\frakell)^2$.  A proof of this equality when $\End A$ is maximal is given in \cite[Theorem 3.15]{waterhouse} and the proof generalizes immediately.
    The action of $\OO_F$ on $A[\frakell]$ induces an action of the field $\OO_F/\frakell$.  In particular, $A[\frakell]$ is an $\OO_F/\frakell$-vector space of dimension 2. Therefore proper $\OO_F$-stable subgroups are precisely the $1$-dimensional $\OO_F/\frakell$-subspaces of $A[\frakell]$.  Thus, the size of the kernel of an $\frakell$-isogeny is $\#(\OO_F/\frakell) = N(\frakell)$. For separable isogenies, this is the same as the degree of the isogeny.
    \end{proof}
 
Notice that the definition of an $\frakell$-isogeny generalizes the definition of an $\ell$-isogeny for elliptic curves.  This definition is useful because of the following theorems about $\frakell$-isogeny graphs, which mirror the foundational theorems used in Bisson and Sutherland's original algorithm for elliptic curves \cite[\S2.1]{bisson-sutherland}.  Write $\OO(\ff^+)$ for the order of $K$ defined by the ideal $\ff^+\subseteq \OO_F$ and $\cl(\ff^+)$ for the class group of $\OO(\ff^+)$.    Denote by $\textnormal{Ab}_{\pi, \ff^+}$ the set of abelian varieties defined over $\FF_q$ in the isogeny class defined by $\pi$ whose endomorphism ring has identifying ideal $\ff^+$.  Given an ideal $\frakell\subseteq \OO_F$, define the symbol $(K/\frakell)$ to be $1, 0$, or $-1$ when $\frakell$ is split, ramified, or inert in $K$, respectively.

 \begin{theorem}
 \label{action}
  There is a faithful action of $\cl(\ff^+)$ on $\textnormal{Ab}_{\pi, \ff^+}$, where an ideal lying over a prime $\frakell\subseteq \OO_F$ acts by an $\frakell$-isogeny.
  \end{theorem}
   \begin{proof} 
  The action of the class group is a classical result \cite{shimura-taniyama, waterhouse}. The fact that these isogenies are $\frakell$-isogenies is clear from the definition, as pointed out in \cite[Theorem 4.3]{bjw}.
   \end{proof}

 \begin{theorem}
 \label{isograph}
 Let $\frakell$ be a prime of $\OO_F$ which does not divide $\vv = \ff^+(\OO_F[\pi])$, and let $A\in \textnormal{Ab}_{\pi, \ff}$.  Then there are exactly $1 + (K/\frakell)$ $\frakell$-isogenies starting from $A$ which lead to varieties with endomorphism ring isomorphic to $\OO(\ff)$, and these are the only varieties defined over $\FF_q$ which are $\frakell$-isogenous to $A$.
  \end{theorem}
   \begin{proof} 
   This follows immediately from Theorem 4.3 and Proposition 4.10 in \cite{bjw}.
   \end{proof}

 \subsection{Computing Isogenies}
 Given an ideal $\aa\subseteq \End A$, we must determine the isogeny $\phi_\aa$ which corresponds to the action of $\aa$.  One option would be to simply compute all possible isogenies, in the style of \cite{bisson-sutherland}.  However, we do not have the benefit of an easily calculable modular polynomial when the dimension of $A$ is $g > 1$. Instead, we compute the target of an isogeny from its kernel.   We can determine the $\ker\phi_\aa = \cap_{a\in \aa}A[a]$ in the same way as Bisson \cite{bisson}.  If $\LL\subseteq \OO$ is a prime ideal lying over $\ell$ which does not divide $[\OO_K : \ZZ[\pi]]$, then we can write $\LL = \ell\OO + r(\pi)\OO$ where $r(x)$ is a factor modulo $\ell$ of the characteristic polynomial $f_\pi$ of Frobenius.
   Then the kernel of $\phi_\LL$ is simply the kernel of the isogeny $r(\pi)$ restricted to $A[\ell]$.
 
 After obtaining the subgroup corresponding to the given ideal, we must compute the target variety.  This difficulty will be one of the main bottlenecks of the algorithm.  We will only consider computing isogenies for the case $g = 2$.  Consider a prime $\frakell\subseteq \OO_F$ lying over $\ell$, and let $\phi$ be an $\frakell$-isogeny whose kernel is known. There are two cases, depending on the norm of the ideal $\frakell$. If  $N\frakell = \ell^2$, then $\phi$ is commonly called an $(\ell,\ell)$-\emph{isogeny}, and the target variety can be determined by the algorithm in \cite{cr}.  If $N\frakell = \ell$, then $\phi$ is known as a \emph{cyclic isogeny} and  the target variety can be determined by the algorithm in \cite[Chapter 4]{dudeanu} and \cite{ djrv}.  In both cases, the algorithm is polynomial in $\ell$ and $\log q$.
By an algorithm of Mestre \cite{mestre}, we assume that the output of these algorithms is the Jacobian of a curve defined over the minimal field; see also \cite{cr}. By combining these algorithms and analyzing the running times of each algorithm, we find the following theorem.

\begin{theorem}
\label{iso-compute} 
Let $A$ be an ordinary, principally polarized abelian variety defined over $\FF_q$ of dimension $g = 2$ with maximal RM.  Write $K = \QQ(\pi)$, where $\pi$ is Frobenius, and let $F$ be the maximal totally real subfield of $K$.    Let $\frakell\subseteq \OO_F$ be a prime lying over $\ell$ and assume the following:
\begin{enumerate}
	\item $\frakell = \beta\OO_F$ for a totally positive element $\beta$.
	\item The index $[\OO_F : \ZZ[\pi+\overline\pi]]$ is not divisible by $2\ell$.
\end{enumerate}¥
If $\ell$ is bounded by $L[1/2,c_0](q)$ for some constant $c_0> 0$, then an $\frakell$-isogeny with a given kernel can be found in time
$
	L[1/2, dc_0](q)
$
where $d > 0$ is a constant.
 \end{theorem}

The extra assumptions in Theorem \ref{iso-compute} are required only for isogenies of degree $\ell$, which are called cyclic isogenies. These isogenies are the hardest case to handle when computing $\frakell$-isogenies.  If the extra assumptions are dropped, then it is possible that the target variety does not admit any principal polarization, which presents a major problem to computing the isogeny.  Bisson avoided this problem by using CM-types and reflex fields to generate relations which are easily computable \cite[\S4]{bisson} . However, it is only known that these easily computed relations are sufficient for determining the endomorphism ring under certain heuristic assumptions \cite[Theorem 7.1]{bisson}.

\subsection{Navigating isogeny graphs and identifying abelian varieties with invariants} \label{inv-pol}
We need a way to identify abelian varieties as we navigate the various  $\frakell$-isogeny graphs with the class group action.  To do this, we follow the ideas of \cite[\S4.2]{jw}.
Rosenhain invariants\footnote{
There are other alternative invariants that can be used in a similar way, such as G\"undlach or Igusa invariants. For our purposes, the choice of invariants used is not important.} 
can be used to identify isomorphism classes of principally polarized simple ordinary abelian varieties of dimension $g= 2$ over $\FF_q$, i.e. pairs $(A,\lambda)$ where $A$ is an abelian variety over $\FF_q$ with a fixed principal polarization $\lambda$.    Meanwhile, $\cl(\OO)$ acts on unpolarized abelian varieties, i.e. abelian varieties $A$ with no fixed polarization.  We say that $A$ is \emph{principally polarizable} if a principal polarization $\lambda$ exists, but is not fixed. By definition, a single isomorphism class of principally polarizable abelian varieties can be partitioned into isomorphism classes of principally polarized abelian varieties according to the different choices for principal polarization.   Thus the isomorphism class of $A$ as an unpolarized principally polarizable abelian variety is uniquely represented by the list of invariants which correspond to the different polarizations of $A$, up to isomorphism.

To investigate how the class group action relates to the different choices of polarizations, it is useful to remember the  \emph{polarized class group}\footnote{Also known as the \emph{Shimura class group}.}, as discussed in \cite{bisson,bjw, jw}.  For a given order $\OO\subseteq K$, the polarized class group is defined as follows.
  $$
  	\frakC(\OO) 
	= \{(\aa,\alpha) 
		\mid \aa\subseteq \OO, 
		\aa\overline\aa = \alpha\OO,
		 \alpha \in F \textnormal{ totally positive}\}/\sim.
$$
Multiplication is performed component-wise, and $(\aa, \alpha)\sim (\bb,\beta)$ if there is an element $u\in K^\times$ such that $\aa = u\bb$ and $\beta = u\overline u \alpha$. Given an order $\OO\subseteq K$ containing $\OO_F$, the structure of $\frakC(\OO)$ can be seen through the following exact sequence
  \begin{equation}
  	1 \to (\OO_F^\times)^+/N_{K/F}(\OO^\times) 
	\xrightarrow{u\mapsto (\OO,u)} \frakC(\OO) 
	\xrightarrow{(\aa, \alpha)\mapsto \aa} \Cl(\OO) 
	\xrightarrow{\aa\mapsto N_{K/F}(\aa)} \Cl^+(\OO_F),
	\label{class-group-seq}
  \end{equation}
  where $(\OO_F^\times)^+$ is the set of totally positive units in $\OO_F$, $N_{K/F}$ is the relative norm from $K$ to $F$, and $\Cl^+(\OO_F)$ is the narrow class group of $F$.
   While the usual class group $\cl(\OO)$ acts on isomorphism classes of abelian varieties, the polarized class group $\frakC(\OO)$ acts on isomorphism classes of {principally polarized} abelian varieties.
The image of $\frakC(\OO)$ inside of $\Cl(\OO)$ is the subgroup of ideals which act freely on the set of principally {polarizable} abelian varieties. By assuming that $\Cl^+(\OO_F)$ is trivial, we ensure that $\frakC(\OO)\to \Cl(\OO)$ is surjective, hence all isogenies arising from the action of $\Cl(\OO)$ preserve principal polarizability, which is necessary for the sake of computing isogenies.   The exact sequence also shows how to count the number of isomorphism classes of principal polarizations.

\begin{prop} 
If $A$ is a principally polarizable ordinary simple abelian variety over $\FF_q$ with maximal RM and $\End A = \OO$, then the group 
$$
	\mathcal U(\OO) := (\OO_F^\times)^+/N_{K/F}(\OO^\times)
$$
 acts freely and transitively on the set of isomorphism classes of principal polarizations of $A$.   When the dimension of $A$ is $g = 2$, then $\mathcal U(\OO)$ is an $\FF_2$-vector space of dimension $0\leq d\leq 1$ and $d$ only depends on $F$ and $K$.
 \end{prop}
  \begin{proof} 
 This Proposition summarizes Proposition 5.4 and Lemma 5.5 of \cite{bjw}.
  \end{proof}
  
 This proposition implies that one may check whether an ordinary principally polarized abelian surface $A$ is fixed under the action of an ideal $\aa$ by simply computing the list of one or two invariants for the target variety with different polarizations, and checking if the invariants of $A$ are on this list.  
  
  There are two practical improvements to mention.  First, if $\LL\subseteq \OO$ corresponds to an $(\ell,\ell)$-isogeny, then one can compute $\phi_\LL$ as polarization-preserving by working with the element $(\LL,\ell)\in \frakC(\OO)$, as in \cite{bisson}. Thus, it is always sufficient to only check one invariant for $(\ell,\ell)$-isogenies. Second, if one is computing a chain of isogenies $\phi_{\LL_1}, \dots, \phi_{\LL_k}$ corresponding to prime ideals $\LL_i$ and $\phi_{\LL_k}$ is  a cyclic isogeny, then we can simply make an arbitrary choice of polarization for the target varieties at each step $\phi_{\LL_1},\dots, \phi_{\LL_{k-1}}$, and compute both polarizations for $\phi_{\LL_k}$, if necessary.   Hence, the issue of polarizations will not change the complexity of any algorithms, and we will let this detail be implicit in the remainder of our exposition.

 \subsection{Small Primes\label{sp}}
We only use the class group action to determine the power of $\pp$ dividing $\ff^+(\End A)$ when $\pp$ is a prime ideal dividing $\vv = \ff^+(\OO_F[\pi])$ which has large norm.  For small primes $\pp$ such that $\pp^k \mid \vv$, one could follow Bisson \cite{bisson} and use the method of Eisentr\"ager and Lauter \cite{eis-lauter} to test whether $\End A$ contains $\OO(\pp^k)$.  However, this method does not immediately produce an isogenous abelian variety $A'$ such that $\ff^+(\End A)$ is not divisible by the small prime factors of $\vv$, which is a feature in Bisson and Sutherland's original elliptic curve algorithm \cite{bisson-sutherland}.  For this, we need the following additional result about isogeny graphs. More background may be found in \cite{bjw}.

\begin{theorem}[{\cite[Theorem 4.3]{bjw}}]
Let $V$ be any connected component of the $\frakell$-isogeny graph for the isogeny class of an ordinary, absolutely simple abelian variety $A$ with Frobenius $\pi$ and maximal RM by $F$.  

\begin{enumerate}
	\item The graph $V$ consists of levels $\{V_i\}_{i\geq 0}$ such that each level $V_i$ shares a common endomorphism ring $\OO_i$ and the valuation at $\frakell$ of $\ff^+(\OO_i)$ is $i$.  The valuation of $\ff^+(\OO_i)$ at other primes is the same for all $i$, and the number of levels is equal to the valuation at $\frakell$ of $\ff^+(\OO_F[\pi])$.  
	\item  The graph $V$ is an $N(\frakell)$-volcano if and only if $\OO_0^\times \subseteq F$ and $\frakell$ is principal in $\OO_0\cap F$.
\end{enumerate}
\label{iso-graph}
 \end{theorem}
 
 This theorem implies that finding the power of each $\frakell\mid \vv$ which divides $\ff^+(A)$ is equivalent to finding the level of $A$ in the $\frakell$-isogeny graph.  When navigating the graph, repeatedly moving from level $V_{i+1}$ to level $V_{i}$ until reaching $V_0$ is known as \emph{ascending} the graph. 
We know that the $\frakell$-isogeny graph is a volcano in our applications with no additional assumptions because we will assume in our algorithms that $\cl^+(F) = 1$ and $\OO_F^\times = \OO_K^\times$ for the sake of Theorem \ref{iso-compute} and Lemma \ref{biglemma}. Hence the algorithms presented for volcano navigation in \cite{sutherland} immediately generalize and provide a way to find the level of a variety and ascend the graph.
   We call this method \emph{isogeny climbing}, following the terminology given in the elliptic curve case. Isogeny climbing allows one to determine the power of $\frakell$ dividing $\ff^+(A)$, and also allow one to find a new abelian variety $A'$ such that $\ff^+(A')$ is the same as $\ff^+(A)$ with a given small prime factor removed. Notice that these methods are only efficient for small primes because we have to compute a large number of isogenies.

\section{Class Group Relations}

Let $A$ be an absolutely simple, ordinary, principally polarized abelian variety with maximal RM, of dimension $g$ over $\FF_q$ with Frobenius $\pi$.  We do not restrict the dimension $g$ or the class group of $F$ for this section because no simplicity is gained by focusing on  special cases. Moreover, we recover many of the statements of \cite{bisson-sutherland} when setting $g = 1$. As before, write $K = \QQ(\pi)$ and let $F$ be the maximal totally real subfield.
We begin by recalling a well-known fact about class groups.
\begin{lemma} 
	If $\OO(\ff_1^+) \subseteq \OO(\ff_2^+)$, then there is a surjective map $\cl(\OO(\ff_1^+)) \to \cl(\OO(\ff_2^+))$ induced by mapping $\aa \mapsto \aa\OO(\ff_2^+)$.
 \end{lemma}
 
  \begin{proof} 
  For any order $\OO\subseteq K$ of conductor $\ff$, let $I_{K,\OO}(\ff)$ be the group generated by ideals of $\OO$ coprime to $\ff$ and let $I_{K}(\ff)$ be the group generated by ideals of $\OO_K$.  There is an isomorphism $I_{K,\OO}(\ff) \to I_K(\ff)$ given by the map $\aa \mapsto \aa\OO_K$ where the inverse is given by $\bb\mapsto \bb\cap \OO$ \cite[Proposition 3.4]{lv-deng}.  This map induces the surjective map of class groups.
  \end{proof}
 
 This map allows us to define the concept of a relation analogously to Bisson \cite{bissonGRH, bisson}.  Note that we do not follow Bisson and Sutherland's definition of relation in \cite{bisson-sutherland} because we do not have the benefit of an easily computable modular polynomial.

\begin{definition} 
	A \textbf{relation} $R$ is a tuple of ideals $(\aa_1, \dots, \aa_k)$ of ideals of $\OO_F[\pi]$.  A relation \textbf{holds} in an order $\OO$ if the product is trivial under the map from the lemma.  Similarly, a relation \textbf{holds} for $A$ if the corresponding composition of isogenies fixes $A$.
 \end{definition}

 Combining the lemma and the definition of relation, we obtain the following key corollary which is analogous to \cite[Corollary 4]{bisson-sutherland}.  This allows us to create an algorithm where testing class group relations determines the prime powers $\pp^k$ which divide $\ff^+(A)$.
 \begin{corollary}
	Let $\vv\subseteq \OO_F$ be an ideal which is divisible by a prime power $\pp^k\subseteq \OO_F$.  Write 
	$$
		\ff_1^+  = \pp^{k - 1 - v_\pp(\vv) } \vv;
	$$
	$$
		\ff_2^+ = \pp^k.
	$$  
	Assume there is some relation $R$  which holds in $\OO(\ff_1^+)$ but not in $\OO(\ff_2^+)$.
	Given any $\ff^+\subseteq \OO_F$ which divides $\vv$, $\pp^k \mid \ff^+$ if and only if the relation $R$ does not hold in $\OO(\ff^+)$.
	\label{maincor}
\end{corollary}
 \begin{proof} 
	If $\pp^k\mid \ff^+$, then $\OO(\ff^+)\subseteq \OO(\ff_2^+)$, so the lemma implies the relation $R$ does not hold in $\OO(\ff^+).$  
	Conversely, suppose $R$ does not hold in $\OO(\ff^+)$.  The lemma implies $\OO(\ff_1^+) \not\subseteq \OO(\ff^+)$.  But $\ff_1$ is the same as $\vv$, except that the power of $\pp$ is decremented.  Because $\ff^+$ divides $\vv$, we deduce that $\pp^k$ divides $\ff^+$.
 \end{proof}
 
\subsection{Review of class field theory}
In order to prove that an algorithm based on Corollary \ref{maincor} is unconditionally correct, we need to find infinitely many relations $R$ that satisfy the necessary conditions.  This is the reason why we use the traditional class group $\cl(\OO)$ instead of the polarized class group $\mathfrak C(\OO)$ used by Bisson.  Specifically, we can use ring class fields and class field theory to understand the structure of $\cl(\OO)$ because prime ideals of $\OO$ are principal if and only if they split completely in the ring class field of $K$; see Corollary \ref{princ=split} below.  

 Throughout this section, $K$ remains a fixed CM-field, as before. Let us review the definitions and notation of class field theory so that we can recall the correspondence between finite abelian extensions of $K$ and certain subgroups of   fractional ideals of $K$.
  We refer the reader to \cite{janusz, lv-deng, neukirch} for additional background on class field theory.
 
   For our purposes, a \emph{modulus} $\mm$ for $K$ is an ideal of $\OO_K$. In general, a modulus can include real infinite primes which correspond to real embeddings, but CM-fields do not have any real embeddings, hence we will ignore the infinite primes in this exposition for simplicity.
Write $I_K$ for the group of all fractional ideals in $K$.  Given a fixed modulus $\mm$, we denote by $I_K(\mm)$ the subgroup of $I_K$ generated by ideals of $\OO_K$ coprime to $\mm$.  
Similarly, let $P_K$ be the group of all principal ideals in $K$. Write
			$P_{K,\OO}(\mm) $ for the group generated by 
			$$
				\{\alpha\OO_K : \alpha\in \OO, \ \ \alpha\OO + \mm = \OO\}
			$$ 
			and $P_{K,1}(\mm)$ for the group generated by 
			$$
				\{\alpha\OO_K : \alpha\in \OO_K, \ \alpha \equiv 1\bmod \mm\OO_K\}.
			$$
			  Note that the definition of $P_{K,1}(\mm)$ must be modified if $\mm$ is divisible by infinite primes, but this is the simplest definition in our case because we will always assume $\mm$ is a product of finite primes; see \cite[Lemma 3.5]{lv-deng}.

A subgroup $H \subset I_K$ is a \emph{congruence subgroup (defined modulo $\mm$)} if
$$
	P_{K,1}(\mm) \subseteq H \subseteq I_K(\mm).
$$
If $\mm$ divides $\mm'$, then $H\cap I_K(\mm')$ is also a congruence subgroup, but it is defined modulo $\mm'$ rather than modulo $\mm$.  In this case, we call $H\cap I_K(\mm')$ a \emph{restricted} congruence subgroup. This leads us to introduce the following equivalence relation. If $H_1$ and $H_2$ are congruence subgroups modulo $\mm_1$ and $\mm_2$, respectively, then we say $H_1$ and $H_2$ are \emph{equivalent} if they have a common restriction, i.e. if there is a modulus $\mm_3$ such that $H_1 \cap I_K(\mm_3) = H_2\cap I_K(\mm_3)$.  In this case, there is an isomorphism $I_K(\mm_1)/H_1 \cong I_K(\mm_2)/H_2$ \cite[Lemma V.6.1]{janusz}.  For each equivalence class $\mathbf H$ of congruence subgroups, there is a unique modulus $\ff$ and a congruence subgroup $H_\ff$ defined modulo $\ff$ such that $H_\ff \in \mathbf H$, and $\ff$ divides the defining modulus of every congruence subgroup in $\mathbf H$ \cite[Lemma V.6.2]{janusz}.  Such an $\ff$ is called the \emph{conductor} of $\mathbf H$.

Let $L$ be a finite abelian extension of $K$, and let $I_K^S$ be the subgroup of $I_K$ generated by prime ideals which do not ramify in $L$.  There is an \emph{Artin map}
$$
	\Phi_{L/K}: I_K^S \to \Gal(L/K)
$$
where a prime $\pp$ is sent to the \emph{Frobenius automorphism} $\sigma_\pp\in \Gal(L/K)$.  Specifically, $\sigma_\pp$ is the unique automorphism such that $\sigma_\pp(x) \equiv x^{N(\pp)} \bmod \mathfrak P$ for all $x\in \OO_L$ for any ideal $\mathfrak P\subseteq \OO_L$ lying over $\pp$.  The map is extended to $I_K^S$ multiplicatively.  
  Write $\Phi_{L/K, \mm}$ for the restriction of $\Phi_{L/K}$ to the subgroup ${I_K(\mm)}$ where $\mm$ is a modulus divisible by all primes of $K$ which ramify in $L$.  We say that the \emph{reciprocity law holds} for $(L, K, \mm)$ if $\ker \Phi_{L/K} \supseteq P_{K,1}(\mm)$. In this case, $\ker \Phi_{L/K,\mm}$ is a congruence subgroup defined modulo $\mm$.

\begin{definition} 
Let $L$ be an abelian extension of $K$. Write $\mathbf H(L/K)$ for the equivalence class of all congruence subgroups $H_\mm(L/K) = \ker \Phi_{L/K,\mm}$ where $\mm$ is modulus such that the reciprocity law holds for $(L,K,\mm)$.  The \emph{class field theory conductor} $\ff(L/K)$ of the extension $L/K$ is the conductor of the equivalence class $\mathbf H(L/K)$, i.e. the greatest common divisor of all moduli defining congruence subgroups in $\mathbf H(L/K)$.
 \end{definition}
 
 \begin{theorem}[The Classification Theorem, {\cite[Theorem V.9.9]{janusz}}]
 \label{class-thm}
The correspondence $L~\mapsto~\mathbf H(L/K)$ is a one-to-one, inclusion-reversing correspondence between finite abelian extensions $L$ of $K$ and equivalence classes of congruence groups of $K$.
  \end{theorem}

Lv and Deng showed the following consequences of the classification theorem.

\begin{theorem}[{\cite[Theorem 4.2]{lv-deng}}]
\label{ring-class-field}
Let $\OO\subseteq K$ be an order with conductor $\ff$. Then there exists a unique abelian extension $L$ of $K$ such that all primes of $K$ ramified in $L$ divide $\ff$, and the Artin map 
$$
	\Phi_{L/K, \ff}: I_K(\ff) \to \Gal(L/K)
$$
satisfies $\ker\Phi_{L/K, \ff} = P_{K,\OO}(\ff)$, providing an isomorphism
$$
	\Cl(\OO) \cong \Gal(L/K)
$$
The field $L$ is called the ring class field of $\OO$.  In the case where $\OO = \OO_K$, the ring class field coincides with the Hilbert class field of $K$, which is the maximal abelian unramified extension of $K$.
\end{theorem}

By observing the basic properties of the Artin map, this theorem provides the following corollary.

\begin{corollary}[{\cite[Corollary 4.4]{lv-deng}}] 
\label{princ=split}
Let $\OO\subseteq K$ be an order. If $\pp\subseteq\OO$ is a prime ideal coprime to $\ff$, then $\pp$ is principal if and only if $\pp$ splits completely in the ring class field of $\OO$.
 \end{corollary}
 
 We conclude this section by making one final observation about class field theory conductors.
 
 \begin{lemma} 
 If $L_1$ and $L_2$ are finite abelian extensions of $K$ such that $L_2 \subseteq L_1$, then the class field theory conductor $\ff(L_2/K)$ divides  $\ff(L_1/K)$.
 \label{conductor-division}
  \end{lemma}
  
   \begin{proof} 
 By definition, $\ff(L_2/K)$ divides every modulus for which a congruence subgroup in $\mathbf H(L_2/K)$ is defined, so we simply need to show that there is a congruence subgroup in $\mathbf H(L_2/K)$ defined modulo $\ff(L_1/K)$.  This is equivalent to showing that $\ker\Phi_{L_2/K}$ contains $P_{K,1}(\ff(L_1/K))$.
  This is easy to see by the inclusion-reversing correspondence of Theorem \ref{class-thm} and the definition of $\ff(L_1/K)$, which imply
  $$
  	\ker\Phi_{L_2/K}
		\supseteq \ker \Phi_{L_1/K} 
		\supseteq P_{K,1}(\ff(L_1/K)).
$$
   \end{proof}

\subsection{Existence of relations}
Using the ring class fields defined in the preceding section, we will now find infinitely many relations sufficient for Corollary \ref{maincor}.  To begin, we prove the following technical lemma.

\begin{lemma}
Let $\OO(\ff^+)$ be the order of $K$ which contains $\OO_F$ and corresponds to the ideal $\ff^+\subseteq \OO_F$.  The ring class field $L$ of $K$ of $\OO(\ff^+)$ is a Galois extension of $F$.
\end{lemma}
 \begin{proof} 
	The proof is analogous to the proof of Lemma 9.3 in \cite{cox}.  Because $K/F$ is an imaginary quadratic extension, its Galois group is generated by complex conjugation, which we denote by $\tau$.  Hence showing that $L/F$ is Galois is equivalent to showing that $\tau(L) = L$.  Theorem~\ref{class-thm} states that there is an inclusion-reversing one-to-one correspondence between equivalence classes of congruence subgroups and abelian extensions of $K$.  Thus, we simply need
$
	\ker(\Phi_{\tau(L)/K, \ff}) 
		= \ker(\Phi_{L/K, \ff})
$
where $\ff = \ff^+\OO_K$.  Notice that $\tau(\ff) = \ff$ because $\ff^+$ is an ideal in a totally real field, and $\tau(\OO_K) = \OO_K$.

Theorem \ref{ring-class-field} tells us that 
$$
	\ker(\Phi_{\tau(L)/K, \ff}) = P_{K,\OO}(\ff).
$$
 It is easy to see that $\tau(P_{K,\OO}(\ff)) = P_{K,\OO}(\ff)$ because an ideal $\aa$ is coprime to $\ff$ if and only if $\tau(\aa)$ is coprime to $\tau(\ff) = \ff$. We have $\ker(\Phi_{\tau(L)/K, \ff}) = \ker(\Phi_{\tau(L)/\tau(K), \ff})= \tau(\ker(\Phi_{L/K, \ff}))$ by definition because $\tau(K) = K$. Therefore
 $$
 	\ker(\Phi_{\tau(L)/K, \ff})
 		=\tau(\ker(\Phi_{L/K, \ff})))
		= \tau(P_{K,\OO}(\ff))
		= P_{K,\OO}(\ff)
		= \ker(\Phi_{L/K, \ff}),
 $$
 proving that $\tau(L)$ and $L$ corresponding to the same congruence subgroup, as desired.
 \end{proof}
 
 Now we can prove the existence of the needed relations, using the same idea as Bisson and Sutherland.
 
 \begin{prop}
 \label{relations-exist}
	Assume that $\OO_K^* = \OO_F^*$. Let $\vv\subseteq \OO_F$ be an ideal which is divisible by $\pp^k\subseteq \OO_F$.   If $N(\pp) \geq 3$, then there are infinitely many relations $R$ satisfying the assumption of Corollary $\ref{maincor}$ above.  
	Specifically, write $\ff^+_1 =  \pp^{k - 1 - v_\pp(\vv) } \vv$ and $\ff^+_2 = \pp^k$.  
	Then there are infinitely many primes $\frakell \subseteq \OO_F$ which split in $\OO_K$ such that $R$ holds in $\OO(\ff_1^+)$ but not $\OO(\ff_2^+)$, where $R$ is the one-element relation consisting of any prime lying over $\frakell$.
\end{prop}

\begin{proof}
Let $S_1$ and $S_2$ be the set of primes of $\OO_F$ which split into principal ideals in $\OO(\ff^+_1)$ and  $\OO(\ff^+_2)$, respectively. Our goal is to show that $S_1\setminus S_2$ is infinite.  This proves the claim because every prime of $\OO_F[\pi]$ lying over a prime in $S_1\setminus S_2$ is a relation which holds in $\OO(\ff^+_1)$ but not in  $\OO(\ff^+_2)$, as desired.
	
	 Write $L_1$ and $L_2$ for the ring class fields of $\OO(\ff_1^+)$ and $\OO(\ff_2^+)$, respectively.   An ideal of $\OO(\ff^+_i)$ is principal if and only if it splits completely in $L_i$ by Corollary \ref{princ=split} because $L_i$ is a ring class field.  Hence the sets $S_i$ are the sets of primes in $\OO_F$ which split completely in $L_i$, respectively. 
The Chebotarev Density Theorem implies that $L_2\subseteq L_1$ if and only if $S_1\setminus S_2$ is finite \cite[Proposition VII.13.9]{neukirch}.  But $\pp^k$ divides the conductor of  $\OO(\ff_2)$ and does not divide the conductor of $\OO(\ff^+_1)$ by construction.  Because $N(\pp) \geq 3$, the lemma below implies that $\pp^k$ divides the class field theory conductor $\ff(L_2/K)$ but does not divide $\ff(L_1/K)$.  By Lemma \ref{conductor-division} we have $L_2\not\subseteq L_1$, which completes the proof.
\end{proof}

The following lemma and its proof are a generalization of \cite[Exercises 9.20-22]{cox}.
\begin{lemma}
Let $K$ be a CM field with totally real subfield $F$, and let $\OO\subseteq K$ be an order containing $\OO_F$.
Suppose that 
			$\OO_K^* = \OO_F^*$.
If $L$ is the ring class field of $\OO$, 
then the conductor ideal $\ff$ of $\OO$ and the conductor ideal $\ff(L/K)$ of the ring class field $L$ may only differ by primes of $K$ of norm $2$.
\label{biglemma}
\end{lemma}

 \begin{proof}
	By definition, the class field theory conductor $\ff(L/K)$ is a divisor of the conductor ideal $\ff$ because $\ff$ is the modulus used to define the extension $L/K$.  Thus, if $\ff(L/K) \neq \ff$, then we may write $\ff = \mathcal P\mm$ where  $\mathcal P$ is a prime ideal of $K$ and $\ff(L/K)$ divides $\mm$.  Assume that this is the case for some prime $\mathcal{P}$ with $N(\mathcal{P}) \neq 2$. We will find a contradiction.

	  According to Theorem \ref{class-thm}, there is a equivalence class $\mathbf H(L/K)$ of congruence subgroups defined with various moduli corresponding to the extension $L/K$.   Because $\ff(L/K)$ divides $\mm$, there is a congruence subgroup $\ker\Phi_{L/K,\mm}$ defined modulo $\mm$. By construction, $P_{K,\OO}(\ff)$ is the restriction of $\ker\Phi_{L/K,\mm}$ to $I_K(\ff)$, i.e.
	$
		\ker(\Phi_{L/K,\mm}) \cap I_{K}(\ff)=  P_{K,\OO}(\ff).
	$ 
	This proves that
	\begin{equation}
	\label{set-containment}
		P_{K,1}(\mm) \cap I_{K}(\ff) \subseteq P_{K,\OO}(\ff).
	\end{equation} 
	because $\ker(\Phi_{L/K,\mm})\supseteq P_{K,1}(\mm)$.
	
	By inspecting definitions, the sequence below is exact.
	$$
		\OO_K^* \to (\OO_K/\ff\OO_K)^* \xrightarrow{\phi} P_K \cap I_{K}(\ff)/P_{K,1}(\ff) \to 1.
	$$
	Define  
		$\pi: (\OO_K/\ff)^* \to (\OO_K/\mm)^*$ and
		  $\beta: (\OO_F/\ff^+)^* \to(\OO_K/\ff\OO_K)^* $
		 as the natural maps induced by quotients.  Notice that $\pi$ is surjective and $\beta$ is injective.
		 One may show that 
		 $$
		 	{ \OO_K^*\cdot \ker\pi = \phi^{-1}(I_K(\ff) \cap P_{K,1}(\mm))}
			$$
		and similarly  
		$$
			 {\OO_K^*\cdot\text{Im}\beta = \phi^{-1}(P_{K,\OO}(\ff))},
		$$
		 which proves that ${\ker \pi \subseteq (\OO_K)^*\cdot  \text{Im}\beta}$  by containment \eqref{set-containment} above. Since $\OO_K^* = \OO_F^*$ and $\text{Im}\beta$ is closed under the action of $\OO_F^*$, this shows
	$$
		\ker \pi \subseteq \text{Im} \beta.
	$$
	
	Recall that if $\aa\subseteq \OO_K$ is an ideal, then
			 $$
			 	|(\OO_K/\aa)^*| 
					= N(\aa) \prod_{\substack{\qq\mid\aa\\ \text{prime}}} \left(1 - \frac{1}{N(\qq)}\right)
			$$ 
			This implies that
		$$
			|\ker(\pi)| 
				= \frac{|(\OO_K/\ff)^*|}{|(\OO_K/\mm)^*|}
				= \begin{cases}
					N(\mathcal{P}) & \text{ if } \mathcal{P}\mid \mm;\\
					N(\mathcal{P}) - 1 & \text{ if } \mathcal{P}\nmid \mm.
			\end{cases}
		$$
		
		We will now conclude the proof by showing that $|\ker\pi| = 1$, which is only possible  if $N(\mathcal{P}) = 2$ and $\mathcal{P}\nmid \mm$, according to the formula for $|\ker\pi|$.  
		To prove this, we consider two cases.  
		Writing $ \pp = \mathcal{P}\cap \OO_F$, either $  \pp\OO_K = \mathcal{P}\overline{\mathcal{P}}$ or $  \pp\OO_K = \mathcal{P}$ .  
		In the former case, one observes that $\pi\circ \beta$ is injective, hence $|\ker\pi \cap \text{Im} \beta| = |\ker\pi| = 1$.  
		Clearly this is only possible if $N(\mathcal{P}) = 2$ and $\mathcal{P}\nmid \mm$.
		
		Now consider the latter case, and suppose $  \pp\OO_K = \mathcal{P}$.  Because $ \pp$ is an inert prime of $\OO_F$ and $\mathcal{P}$ divides $\ff$, it follows that $ \pp$ divides $\ff^+$.  Write $\ff^+ =   \pp \mm_0$ for some $\mm_0 \subseteq \OO_F$.  Consider the (not necessarily exact) diagram
				\begin{displaymath}
			    \xymatrix{
				(\OO_F/\ff^+)^*  \ar[d]_\theta 
				\ar[r]^\beta  & (\OO_K/\ff\OO_K)^* \ar[r]^\pi & (\OO_K/\mm)^*\\
			                    (\OO_F/\mm_0)^*           &  
			                   }
			\end{displaymath}			
			 One shows $\ker \pi = \ker\pi\cap \text{Im} \beta \cong \ker \theta$.  However, we can compute $\ker(\theta)$ in the same way that we computed $\ker(\pi)$ to find that
		 	$$|\ker(\theta)|
				= \frac{|(\OO_F/\ff^+)^*|}{|(\OO_F/\mm_0)^*|}
				 = \begin{cases}
					N(  \pp) & \text{ if }   \pp\mid \mm_0;\\
					N(  \pp) - 1 & \text{ if }   \pp\nmid \mm_0.
			\end{cases}$$
	Because $  P$ is inert, $N(\mathcal{P})  = N(  \pp\OO_K) > N( \pp)$.  Hence $|\ker \pi| = |\ker\theta|$ is impossible, which gives the final contradiction.
 \end{proof}

 \section{ Algorithms}
 \label{alg-sec}
 Now that we have proven that there are sufficiently many class group relations for determining the large prime factors of $\ff^+(A)$, we are able to present generalizations of the elliptic curve algorithms in \cite{bisson-sutherland}.
Even though the results of the previous section apply whenever $A$ is an ordinary simple abelian variety of arbitrary dimension with maximal RM, there are no known results for computing an arbitrary isogeny in such generality. We restrict our attention to a manageable case with the following requirements so that we can compute all isogenies arising from the action of the ideal class group.
 
  \subsection*{Requirements \textbf{(R)}}
 For the remainder of the paper, we will focus on the case where the ordinary, absolutely simple, principally polarized abelian variety $A$ has dimension $g = 2$.  We also assume the following are all true, which summarizes the hypotheses found in Theorems \ref{ideal}, \ref{iso-compute} and \ref{iso-graph}, and Lemma \ref{biglemma}. In particular, $\End A$ is uniquely identified by an ideal $\ff^+(A) \subseteq \OO_F$ and we can compute the isogenies corresponding to the action of $\cl(\End A)$.  Notice that the conditions below are all easily verifiable.
 
 \begin{enumerate}
	 \item $A$ has maximal real multiplication by $F$.
	 \item $\OO_K^* = \OO_F^*$.
	 \item $F$ has narrow class number 1.
	 \item The conductor gap $[\OO_F : \ZZ[\pi + \overline \pi]]$ is not divisible by 2.
 \end{enumerate}
 Notice that that the second assumption is very mild.  It is shown in \cite[Lemma II.3.3]{streng} that $\OO_K^* = \OO_F^*$ is true for every primitive quartic CM field except the cyclotomic field $K = \QQ(\zeta_5)$.  
 
 \subsection*{Heuristic assumptions (\textbf{H})}
  We collect the following heuristic assumptions that we require, and we denote all statements requiring these assumptions will be decorated by the letter (\textbf{H}).  The first three assumptions allow us to compute a class group via \cite{bf}, while the final two are moderate heuristic assumptions, similar to those in \cite{bisson-sutherland}, which are only needed to bound the expected run time of the algorithms. In particular, the latter three assumptions are not needed in order to prove the algorithm is correct.  All running times will be analyzed with the subexponential function
  $$
  	L[a,c](n) = \exp\left( (c+o(1))\log(n)^a(\log\log n)^{1-a}\right)	
  $$
  \begin{enumerate}
	 \item \textbf{GRH} is true.
	 \item \textbf{Smoothness assumption.}	The probability $P(\iota, \mu)$ that an ideal of $\OO$ of norm bounded by $e^\iota$ is a power-product of prime ideals of norm bounded by $e^\mu$ satisfies
	$
		P(\iota, \mu) \geq \exp(-u\log u(1 + o(1)))
	$
	for $u = \log(\iota)/\log(\mu)$.
	\item \textbf{Spanning relations} If $\{\pp_1, \dots, \pp_N\}$ is a factor base of ideals generating $\Cl(\OO)$ where  $N = L[a,c_1](\Delta)$ for some $0 < a < 1$ and $c_1 > 0$, then it suffices to collect $N'$ relations to generate all possible relations if $N'/N = L[b,c_2](\Delta)$ for $0< b< a$ and $c_2 > 0$.
	\item \textbf{Random Norms:} Assume that the norms of the ideals generated in the reduction step of \textsc{FindRelation} have approximately the distribution of random integers in $[1,n]$.  This allows us to analyze the probability that a norm is $B$-smooth.
	 \item \textbf{Random relations.} If $\OO_1$ and $\OO_2$ are as in Corollary \ref{maincor} with sufficiently large discriminants, then the relation for $\OO_1$ generated in Find Relation algorithm does not hold in $\OO_2$ with probability bounded above 0.
	\item \textbf{Integer factorization and smoothness checking.}  The number field sieve with expected running time $L[1/3,c_f](n)$ by \cite{blp} and ECM finds a prime factor $p$ of integer $n$ in expected time $L[1/2,2](p)\log^2 n$ by \cite{lenstra}.	
 \end{enumerate}

 \subsection{Finding Relations}
 \label{findrelationsubsec}
 Following \cite{bissonGRH,bisson, bisson-sutherland}, we will use the reduction of random ideals to produce relations which hold in a given order $\OO\subseteq K$ of discriminant $\Delta = \disc(\OO)$. This is the same idea that is used in ideal class group computation algorithms, such as \cite{bf}. For background on ideal reduction, see \cite[Chapter 5]{thiel}.  Given an ideal $\aa\subseteq \OO\subseteq K$, reduction outputs an ideal
  $\bb$ which is equivalent to $\aa$ in $\Cl(\OO)$ such that $N\bb \leq \Delta^2$.  By taking $N\aa > \Delta^2$, we ensure that $\aa\bb^{-1}$ is a nontrivial relation which holds in $\OO$. The reduction algorithm is polynomial in $\log|\Delta|$ for fields of fixed degree.

To test whether this given relation also holds in a second order, we need to solve the principal ideal problem.  As shown in \cite[\S4.3]{biasse} and \cite{bf}, we can solve the principal ideal problem in subexponential time by  using the heuristic assumptions \textbf{(H)}. In practice, one should compute the class group once and for all at the beginning of the algorithm, then simply do reductions as necessary when checking relations.
To give some flexibility in applications, we use a parameter $\mu > 0$ which can be chosen arbitrarily.
 
  \begin{center}
  \begin{algorithm}
\SetAlgoLined
\Input{Orders $\OO_1$ and $\OO_2$}
\Output{Relation $R$ which holds in $\OO_1$ but not $\OO_2$}

Set  $B = L[1/2, \mu](n)$, $D_1 = \disc(\OO_1)$ and $n = |D_1|^2$.

Pick $x_\LL$ such that $x_\LL \leq B/N(\LL)$ for prime ideals $\LL$ of norm bounded by $B$ such that at most $k_0$ are nonzero and $\prod N(\LL) > n$.

Compute the reduced ideal $\bb =  \prod_\LL \LL^{y_\LL}$ of $\aa = \prod \LL^{x_\LL}$.

\If{ $N(\bb)$ is a $B$-smooth integer and the number of nonzero $y_\LL$ is at most $8\log(|D_1|)^{1/2}$}{
	Let $R$ be the relation $(\LL^{x_\LL - y_\LL})_{N\LL < B}$.
	
	\If{$R$ does not hold in $\OO_2$}{
		\Return $R$.
	}
}

Go to Step 2.
 
 \caption{\textsc{FindRelation}}
\end{algorithm}
\end{center}

Notice that a prime ideal $\LL$ is inverse to its complex conjugate $\overline \LL$ because the totally real subfield $F$ has class number 1. Hence we can ensure that at most one of $\LL$ and $\overline \LL$ appears in the relation $R$. In particular, if $x_\LL - y_\LL < 0$ for any prime ideal $\LL$, then replace the prime power $\LL^{x_\LL - y_\LL}$ in the relation $R$ with $\overline \LL^{y_\LL - x_\LL}$, which is an equivalent ideal in the class group.
Implicitly, we throw out every relation which includes an undesirable prime, i.e. a prime ideal of $\OO_F$ that divides the ideal $\ff^+(\OO_F[\pi])$, or the index $[\OO_F : \ZZ[\pi+\overline\pi]]]$, or the index $[\OO_K : \ZZ[\pi]]$.  There are only $O(\log q)$ such primes, so this does not change the complexity of the algorithm.

\begin{prop}[\textbf{H}]
\label{findspeed}
Given orders $\OO_1$ and $\OO_2$ of discriminants $D_1$ and $D_2$, the algorithm \noindent\textsc{FindRelation} has expected running time
$$
	L[1/2, 1/\mu\sqrt{2}](|D_1|) + \log |D_1|^{1 + \epsilon}L[1/2, c](|D_2|)
$$
where $\mu > 0$ is a parameter that can be chosen arbitrarily, $\epsilon$ is arbitrarily small, and $c$ is the constant from the running time bound of principal ideal testing.
 The output relation $R = (\LL_1^{e_1}, \dots, \LL_k^{e_k})$ has 
exponents $e_i$ bounded by $ L\left[1/2, \mu\sqrt{2}\right](|D_1|)$ 
and $k$ bounded by $\frac{16\sqrt 2}{\mu}(\log|D_1|)^{1/2}$.  The prime numbers $\ell_i$ lying under the $\LL_i$ are bounded by $L\left[1/2, \mu\sqrt{2}\right](|D_1|)$.
 \end{prop}
 
  \begin{proof} 
Recall that 
$
	B 
		= L[1/2,\mu](n)
		=L[1/2, \mu\sqrt 2](|D_1|)
$
is the smoothness bound on the norms of the primes where  $n = |D_1|^2$.  Notice
$
	\log n/\log B = \frac{1}{\mu}(\log n)^{1/2}(\log\log n)^{-1/2}.
$
  With the bound $1/\rho(z)^{-1} = z^{z + o(1)}$ for the Dickman function $\rho(z)$ and the same argument as \cite[Proposition 6]{bisson-sutherland}, we expect the number of attempts required to find an ideal $\bb$ with $B$-smooth norm to be be asymptotically bounded by $\rho(\log n/\log B)^{-1} = L[1/2, 1/\mu\sqrt2](|D_1|)$, and a $B$-smooth integer in $[1,n]$ is expected to have $(2 + o(1))\log n/\log B$ distinct prime factors.  Thus we expect the number of prime ideals $\LL$ appearing in step 4 to be at most
$$
	k_0 + 8\log n/\log B \leq k_0 + \frac{8\sqrt{2}}{\mu}(\log |D_1|)^{1/2}(\log\log|D_1|)^{-1/2}
$$
By heuristic assumption \textbf{(H)}, elliptic curve factorization \cite{lenstra} identifies a $B$-smooth integer in time $L[1/2,2](B) = L[1/4, \sqrt{2\mu}](n)$ with high probability.  Therefore, $L[1/2, 1/\mu\sqrt{2}](|D_1|) $ is a bound on the amount of time spent finding a relation which holds in $\OO_1$ with $B$-smooth norm.

Step 6 is solving the Principal Ideal Problem in $\OO_2$, which can be done by following Biasse's algorithm \cite[Algorithm 7]{biasse}. First, one computes the class group in time $L[1/2, c](|D_2|)$ by \cite[Theorem 6.1]{bf}. This involves finding a set of relations which span all relations on a set of primes $\{\mathcal P_1, \dots, \mathcal P_M \}$ which generate $\cl(\OO_2)$.  Then, each ideal $\mathfrak L$ appearing in the relation $R$ is reduced, if necessary, to an equivalent product over the generating set, i.e. $\mathfrak L = (\alpha)\mathcal P_1^{e_1}\dots \mathcal P_M^{e_M}$ for some $\alpha \in K$ and $e_i \geq 0$.  Each reduction takes time 
$
	\log(N(\mathfrak L))^{1+o(1)}L[1/2, c_1](|D_2|)
$
for some constant $c_1 > 0$ by \cite[Proposition 3.1]{biasse}.
Recall that the ideals are bounded by 
$$
	\log N(\mathfrak L) 
		< \log B 
		\leq \log L[1/2, \mu](|D_1|^2)
		\ll  (\log|D_1|)^{1/2} (\log\log|D_1|)^{1/2}
$$
By combining this bound with the bound on the number of primes appearing in $R$, the cost of solving the principal ideal problem is bounded, for some constant $c > 0$, by
$
	{\log |D_1|^{1 + \epsilon}L[1/2, c](|D_2|)}.
$
 We succeed in finding a relation which does not hold in $\OO_2$ within $O(1)$ tries because \textbf{(H)} assumes that the probability of success is bounded above 0. Thus, we obtain the final bound on the running time.
The bounds on $k$, $\ell_i$ and $e_i$ follow directly from the construction of the relation in the algorithm.
  \end{proof}

\subsection{Computing from above}

 \begin{center}
  \begin{algorithm}
  \label{compute-from-above}
\SetAlgoLined
\Input{Abelian variety $A$ satisfying requirements \textbf{(R)}}
\Output{The ideal $\ff^+\subseteq \OO_F$ identifying $\End A$}

Determine the ideal $\vv\subseteq \OO_F$ defining the order $\OO_F[\pi]\subseteq \OO_K$.

Fix bound $C \geq 3$ and initiate $\uu = 1$.

\For{every $\pp\subseteq\OO_F$ with $N\pp < C$ which divides $\vv$}
{Use isogeny climbing to determine $v_\pp = v_\pp(\ff^+(A))$.

Replace $\uu$ with the product $\uu\pp^{v_\pp}$.

Update $A$ to be the abelian variety found by isogeny climbing 

Remove powers of $\pp$ from $\vv$.}

\For{every $\pp\subseteq\OO_F$ with $N\pp \geq C$ which divides $\vv$}
{Find relation $R$ which holds in $\OO(\vv/\pp)$ but not $\OO(\pp)$.\label{for-loop}

\If{ $R$ does not hold in $\OO(\ff^+(A))$}{
Update $\uu$ to be the product $\uu\pp$}}

\Return $\ff^+(A) = \uu$.
 
  \caption{{Computing $\ff^+(A)$ from above}}
\end{algorithm}
\end{center}

Now we present our algorithm for computing $\ff^+(A)$.  As noted in Section \ref{sp}, we can take care of all ``small" prime factors, i.e. primes $\pp$ with $N(\pp) \leq C$ for some $C \geq 2$, by ``isogeny climbing".  The bound $C$ can be chosen arbitrarily  In the presentation of this and all remaining algorithms, we assume that $\ff^+(A)$ is not divisible by the square of any large primes, although the algorithms can be easily modified to handle this possibility. For example, we can modify Algorithm \ref{compute-from-above} by finding relations in Step \ref{for-loop} corresponding to $\pp^k$ for every $k\geq 1$ such that $\pp^k \mid \vv$.  Instead, we simplify the presentation by only checking $k = 1$.
The correctness of this algorithm immediately follows immediately from Corollary \ref{maincor}. There are infinitely many class group relations $R$ as required in Step \ref{for-loop} according to Proposition \ref{relations-exist}.

\begin{prop}[\textbf{H}]
	Algorithm \ref{compute-from-above} has expected running time 
$$
	L[1/2, 2c](q)+ L[1/2, 2\sqrt{d+1}](q).
$$
where $c$ is the constant from principal ideal testing, $d$ is the constant from Theorem \ref{iso-compute}.
 \end{prop}
  \begin{proof} 
Computing $\vv$ takes polynomial time by Corollary \ref{ideal-alg}, and factoring $\vv$ reduces in polynomial time to factoring its norm $N(\vv)$.  This is done in time $L[1/3,c_f](q)$ by \textbf{(H)}.

The number of iterations in the algorithm is the number of primes dividing $\vv$, so there are only $O(\log q)$ iterations needed.  Recall that $\log |\disc(\ZZ[\pi,\overline\pi])| < 4\log q + 12\log 2$ by \cite[Lemma 6.1]{bisson}. Since we only consider orders containing $\OO_F[\pi]\supseteq \ZZ[\pi,\overline\pi]$, this means that \textsc{FindRelation} takes time
$$
	L[1/2, 1/\mu\sqrt{2}](|D_1|) + \log|D_1|^{1 + \epsilon}L[1/2, c](|D_2|)
		= L[1/2, \sqrt 2/\mu](q) + L[1/2, 2c](q)
$$
 for each choice of $\OO_1$ and $\OO_2$, by the proposition above.

Computing whether a relation $R = (\LL_1^{e_1}, \dots, \LL_k^{e_k})$ holds for $A$ requires computing $\sum_{i = 1}^k e_i = L[1/2,2\mu\sqrt{2}](q)$ many isogenies which each take time $ L[1/2,2\mu d\sqrt{2}](q)$ to compute by Theorem \ref{iso-compute}.     Thus the expected running time is 
$$
	 L[1/2, \sqrt 2/\mu](q) + L[1/2, 2c](q)+ L[1/2,2\mu(d+1)\sqrt{2}](q).
$$
Solving $\sqrt 2/\mu = 2\mu(d+1)\sqrt{2}$ presents $\mu = 1/\sqrt{2(d+1)}$ as an optimal choice for $\mu$.  Inserting this into the bound above, we obtain the desired bound:
$$
	L[1/2, 2c](q)+ L[1/2, 2\sqrt{d+1}](q).
$$
  \end{proof}

\subsection{Certifying and Verifying}

\begin{center}
  \begin{algorithm}
\SetAlgoLined
\Input{CM field $K$ with maximal totally real subfield $F$, ideals $\uu, \vv\subseteq \OO_F$}
\Output{Certificate $C$}

\For{For every prime $\pp$ with $v_\pp(\vv) - v_\pp(\uu) > 0$}
{Find a Relation $R_\pp$ which holds in $\OO(\uu)$ but not $\OO(\pp)$}

\For{For every prime $\pp$ of $\uu$}
{Find a Relation $R_\pp$ which holds in $\OO(\uu\pp^{-1})$ but not $\OO(\pp)$.}

\Return the certificate $C = (\uu, \vv, K, \{R_\pp\}_{\pp\mid \vv})$.

  \caption{\textsc{Certify}}
\end{algorithm}
\end{center}

The relations generated throughout Algorithm \ref{compute-from-above} work for any abelian variety in the given isogeny class which satisfies the requirements \textbf{(R)}.  Collecting these relations without fixing a specific abelian variety $A$ gives the \textsc{Certify} and \textsc{Verify} algorithms below.
 Indeed, \textsc{Certify} does not require an abelian variety as input; it simply gives a certificate that allows anyone to check the claim that $\ff^+(A) = \uu$ via the subsequent \textsc{Verify} algorithm, up to small prime factors. In the presentation below, we ignore the small prime factors, and assume they are taken care of by isogeny climbing, as before.
The two loops in \textsc{Certify} and \textsc{Verify} correspond to checking that $\uu$ divides $\ff^+(A)$ and $\uu$ is not a proper divisor of $\ff^+(A)$, respectively.  Again, the correctness follows immediately from Corollary \ref{maincor}.

  \begin{corollary}[\textbf{H}]
  	The expected running time of \textsc{Certify} is within a factor of $O(\log q)$ of the expected running time of \textsc{FindRelation}.  If $D_1$ is discriminant of the order corresponding to the ideal $\uu$, then the certificate has size $O(\log|D_1|\log\log|D_1|)$.
   \end{corollary}
    \begin{proof} 
Using the bound $|\disc(\OO_F[\pi])|\leq |\disc(\ZZ[\pi,\overline\pi])| \leq 4^{6}q^4$ from \cite[Lemma 6.1]{bisson}, we find that the number of prime factors of $\vv$ is $O(\log q)$.  Thus, \textsc{FindRelation} is called $O(\log q)$ many times in \textsc{Certify}.

 By the bounds on the $\ell_i$, $k$ and $e_i$, we see that the size of the certificate is $O(\log|D_1|\log\log|D_1|)$.  Here, we use the fact that $N(\LL_i) \leq \ell^4$  and the size of an ideal $\aa$ in a number field of fixed degree is $O(\log N(\aa))$ \cite[Corollary 3.4.13]{thiel}.
    \end{proof}

\begin{center}
  \begin{algorithm}
\SetAlgoLined
\Input{Abelian variety $A$ satisfying requirements \textbf{(R)}, certificate $C = (\uu, \vv, K, \{R_\pp\}_{\pp\mid \vv})$}
\Output{\texttt{true} or \texttt{false}}

\For{For every prime $\pp$ with $v_\pp(\vv) - v_\pp(\uu) > 0$}
{Check that $R_\pp$ holds in $\End A$ by computing isogenies

Immediately \Return \texttt{false} if failure occurs.}

\For{For every prime $\pp$ dividing $\uu$}
{Check that $R_\pp$ does not hold in $\End A$ by computing isogenies.

Immediately \Return \texttt{false} if a relation holds.}

\Return \texttt{true}
 
  \caption{\textsc{Verify}}
\end{algorithm}
\end{center}

 \begin{prop}[\textbf{H}] 
 Given a certificate $C = (\uu, \vv, K, \{R_\pp\}_{\pp\mid \vv})$ produced by \textsc{Certify} with parameter $\mu > 0$ and $A/\FF_q$, the expected run time of \textsc{Verify} is
  $$
  	L[1/2, \mu(d+1)\sqrt{2}](|D_1|)
  $$
 where $d$ is the constant from Theorem \ref{iso-compute} and $D_1$ is the discriminant of the order identified by $\uu$.
  \end{prop}
  
   \begin{proof} 
There are at most $O(\log q)$ relations in a certificate.  By Proposition \ref{findspeed}, each relation has at most  $O(\log^{1/2}|D_1|)$ distinct primes with exponents $e_i$ bounded by $L[1/2,\mu\sqrt 2](|D_1|)$.  The primes $\ell_i$ are bounded by $L[1/2, \mu\sqrt 2](|D_1|)$, which produces the claim by Theorem \ref{iso-compute}.
   \end{proof}

 \subsection*{Remark:} 
 In \cite[Algorithm 2]{bisson-sutherland}, Bisson and Sutherland define an additional algorithm to compute the endomorphism ring of an ordinary elliptic curve by repeatedly calling their versions of the \textsc{Certify} and \textsc{Verify} algorithms.  This algorithm performs well for large orders because their version of \textsc{FindRelation} uses binary quadratic forms to quickly check class group relations in various orders.  In our case, \textsc{FindRelation} is much more costly because no analogue of binary quadratic forms exists for performing computations in the class group of orders in general CM fields.  As a result, the immediate generalization of \cite[Algorithm 2]{bisson-sutherland} which uses the \textsc{Certify} and \textsc{Verify} algorithms above correctly computes $\ff^+(\End A)$, but does not gain an performance improvement by focusing on abelian varieties with nearly-maximal endomorphism ring.

 \section{Computational Example}
 \label{ex-sec}
 We now give an illustrative example to demonstrate how Algorithm \ref{compute-from-above} computes an endomorphism ring by using class group relations.  All computations where performed using Magma \cite{magma} and the AVIsogenies  library \cite{avi} on a 2.3 GHz Intel Core i5 processor.  Although the programs were not optimized for maximum performance, running times are given below to display how isogeny computation is the major bottleneck of the algorithm.

 \subsection{Example}
 Let $q = 82307$ and let $A$ be the Jacobian of the hyperelliptic curve $C$ defined over $\FF_q$ by 
 $$
 	y^2 = x^5 - 3x^4 + 5x^3 - x^2 - 2x + 1.
 $$
 The curve $C$ is the specialization to $s = t = 1$ of the 2-parameter family given in \cite[\S4.3]{gks}.  By \cite[Proposition 3]{gks}, $A$ has maximal real multiplication by $F = \QQ(\sqrt 5)$, which has narrow class number 1.  The characteristic polynomial of the Frobenius endomorphism $\pi$ is
 $$
 	f_\pi(t) = t^4 + 658t^3 + 263610t^2 + 658q t + q^2.
 $$
By \cite[Theorem 6]{hz}, we deduce that $A$ is absolutely simple.

We begin by computing the ideal $\vv\subseteq \OO_F$ which identifies $\OO_F[\pi]$.  Using Algorithm \ref{compute-ideal}, it takes $0.02$ seconds to compute that $\vv = \pp_{11}\pp_{131}$ where $\pp_{11}$ and $\pp_{131}$ are prime ideals of $\OO_F$ of norm $11$ and $131$, respectively.  Therefore, $\End A$ is one of the following four orders, using the notation of Section \ref{identify-orders}:
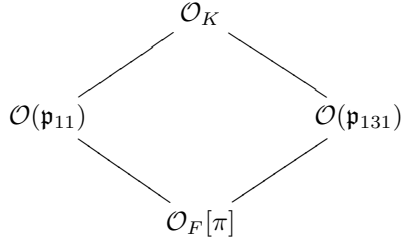
 \begin{displaymath}
    \xymatrix{ 
    			& \OO_K \ar@{-}[dl] \ar@{-}[dr] &\\
    \OO(\pp_{11})\ar@{-}[dr] & & \OO(\pp_{131})  \ar@{-}[dl]\\
				&   \OO_F[\pi] &  }
\end{displaymath}

Next, we need suitable class group relations, as described in Corollary \ref{maincor}.  Consider the prime  number $7$, which is inert in $\OO_F$ and splits in $\OO_K$.  We find 
$$
	f_\pi(t) \equiv (t^2 + t + 6)(t^2 + 6t + 6) \mod 7,
$$
 hence the prime ideals of $\OO_F[\pi]$ lying over $7$ are $\LL_1 = (7, \pi^2 + \pi + 6)$ and $\LL_1' = (7, \pi^2 + 6\pi + 6)$.  It takes $0.22$ and $0.16$ seconds, respectively, to find that the ideal $\LL_1\OO(\pp_{11})$ has order $55$ in $\cl(\OO(\pp_{11}))$, and the ideal $\LL_1\OO(\pp_{131})$ has order 60 in $\cl(\OO(\pp_{131}))$.  Therefore the relation $R_1 = (\LL_1^{55})$ holds in $\OO(\pp_{11})$ but not $\OO(\pp_{131})$. Thus $\pp_{11}$ divides the identifying ideal $\ff^+(\End A)$ if and only if the relation $R_1$ does not hold for $A$.  Similarly, $R_2 = (\LL_1^{60})$ is a relation which holds in $\OO(\pp_{131})$ but not $\OO(\pp_{11})$, hence $\pp_{131}$ divides $\ff^+(\End A)$ if and only if $R_2$ does not hold for $A$.

Finally, we need to compute the chains of isogenies corresponding to the two relations $R_1$ and $R_2$. 
The kernel of the isogeny corresponding to the action of the ideal $\mathfrak{L}_1 = (7, \pi^2 + \pi + 6)$ is the rational symplectic subgroup $G\subseteq A[7]$ whose generators are annihilated by $\pi^2 +\pi + 6$. 
Using AVIsogenies, we enumerate all possible rational symplectic subgroups of $A[7]$, find the desired $G$ in this list, and compute the corresponding $(7,7)$-isogeny.  Continuing in this way, it takes $154$ seconds to compute all 60 isogenies, and we find that neither $R_1$ nor $R_2$ holds for $A$.  By Corollary \ref{maincor}, this implies that $\ff^+(\End A) = \vv = \pp_{11}\pp_{131}$, i.e. $\End A \cong \OO_F[\pi]$.\\

This example demonstrates how our algorithm is more efficient than the algorithm of Eisentr\"ager and Lauter \cite{eis-lauter} in cases when the index $[\OO_K : \OO_F[\pi]]$ is divisible by large primes. Indeed, using Eisentr\"ager and Lauter's algorithm on the example above requires the full $131$-torsion of $A$, which is defined over an extension of degree 17030. As a result, this method is very costly compared to the computation of the $(7,7)$-isogenies performed above.  The same benefit is seen in examples of Bisson's algorithm \cite[\S8]{bisson}, since it also exploits class group relations. Our algorithm differs from Bisson's algorithm by considerably restricting the class of abelian varieties considered, which allows us to avoid certain unconventional heuristic assumptions.

 \bibliographystyle{acm}

\end{document}